\documentclass[11pt]{article}

\usepackage[a4paper, left=3.2cm,top=3cm,right=3.2cm,bottom=3.5cm]{geometry}

\usepackage{concmath}
\usepackage[T1]{fontenc}
\usepackage[utf8]{inputenc}

\usepackage{authblk}
\usepackage{cite}
\usepackage{tabularx}

\usepackage{amsmath,amsthm,amssymb}
\usepackage{dsfont}
\usepackage{mathtools}
\usepackage{subcaption}
\usepackage[color=green!40]{todonotes}
\usepackage{graphicx}
\usepackage{tikz}
\usepackage{bbm}
\usepackage{enumitem}
\usepackage{tikz}
\usepackage{booktabs}
\usepackage{framed}

\usepackage[binary-units]{siunitx}
\usepackage{eurosym}

\usepackage{todonotes}

\newcommand*{\N}{\mathds{N}}

\newcommand*{\R}{\mathds{R}}

\newcommand{\edot}{\, \cdot \, }

\newcommand{\nett}{\mathcal{N}}
\newcommand{\anett}{\mathcal{A}}
\newcommand\pen{\mathcal Q}
\newcommand{\reg}{\mathcal R}

\newcommand{\aval}{c}

\newcommand{\M}{\mathcal M}

\newcommand{\Ro}{\mathbf{R}}
\newcommand{\Ko}{\mathbf{K}}

\newcommand{\Uo}{\mathbf U}

\newcommand{\X}{\mathbb X}
\newcommand{\Y}{\mathbb Y}

\newcommand{\similarity}{\mathcal{D}}

\newcommand{\al}{\alpha}

\newcommand{\encoder}{\mathbf E}
\newcommand{\decoder}{\mathbf D}
\newcommand{\auto}{\mathbf N}

\newcommand{\signal}{x}
\newcommand{\data}{y}

\DeclarePairedDelimiter{\abs}{\lvert}{\rvert}
\DeclarePairedDelimiter{\norm}{\lVert}{\rVert}
\DeclarePairedDelimiter{\innerprod}{\langle}{\rangle}

\DeclareMathOperator{\dom}{dom}

\DeclareMathOperator*{\argmin}{arg\,min}

\DeclareGraphicsExtensions{.eps,.pdf,.png,.jpg}

\newtheorem{theorem}{Theorem}

\theoremstyle{definition}

\newtheorem{proposition}[theorem]{Proposition}
\newtheorem{example}[theorem]{Example}
\newtheorem{cond}[theorem]{Condition}

\newtheorem{definition}[theorem]{Definiton}

\usepackage{soul}
\colorlet{lred}{red!40}
\colorlet{lgreen}{green!40}
\colorlet{lblue}{blue!40}
\definecolor{bananamania}{rgb}{0.98, 0.91, 0.71}

\numberwithin{equation}{section}
\numberwithin{table}{section}
\numberwithin{figure}{section}
\numberwithin{theorem}{section}

\allowdisplaybreaks

\author[1]{Daniel Obmann}
\author[2]{Linh Nguyen}
\author[1]{Johannes Schwab}
\author[1]{Markus Haltmeier}

\affil[1]{Department of Mathematics, University of Innsbruck\authorcr
Technikerstrasse 13, 6020 Innsbruck, Austria\authorcr
 {\tt \{daniel.obmann,johannes.schwab,markus.haltmeier\}@uibk.ac.at}\authorcr\mbox{}}

\affil[2]{Department of Mathematics, University of Idaho\authorcr
Moscow, ID 83844, {\tt lnguyen@uidaho.edu}
 }

\title{Augmented NETT Regularization of Inverse Problems}

\date{February 06, 2021}

\begin{document}
\maketitle
\begin{abstract}
We propose aNETT (augmented NETwork Tikhonov) regularization as a novel data-driven reconstruction framework for solving inverse problems. An encoder-decoder type network defines a regularizer consisting of a penalty term that enforces regularity in the encoder domain, augmented by a penalty that penalizes the distance to the data manifold. We present a rigorous  convergence analysis including stability estimates and convergence rates. For that purpose, we prove the coercivity of the regularizer used without requiring  explicit coercivity assumptions for the networks involved. We propose a possible realization together with a network architecture and a modular training strategy.  Applications to sparse-view and low-dose CT show that aNETT achieves results comparable to state-of-the-art deep-learning-based reconstruction methods. Unlike learned iterative methods, aNETT does not require repeated application of the forward and adjoint models, which enables the use of aNETT for inverse problems with numerically expensive forward models. Furthermore, we show that aNETT trained on coarsely sampled data can leverage an increased sampling rate without the need for retraining. 

\medskip

\end{abstract}

\section{Introduction} \label{sec:introduction}

Various  applications in medical imaging, remote sensing and elsewhere require solving  inverse problems of the form
\begin{equation}\label{eq:ip}
\data^\delta   = \Ko \signal + \eta^\delta \,,
\end{equation}
where $\Ko \colon \X \to \Y$ is an operator between Hilbert spaces modeling the forward problem, $\eta^\delta$ is the data perturbation, $\data^\delta \in \Y$ is the noisy data and $\signal \in \X$ is the sought for signal. Inverse problems are well analyzed and several established approaches for its stable solution exist \cite{EngHanNeu96, scherzer2009variational}. Recently, neural networks and deep learning appeared as new paradigms for solving inverse problems  \cite{arridge2019solving,haltmeier2020regularization,mccann2017convolutional,maier2019gentle,wang2016perspective}. Several approaches based on deep learning  have been developed, including post-processing networks \cite{lee2017deep, jin2017deep, antholzer2018deep,maier2019learning,rivenson2018phase}, regularizing null-space networks \cite{schwab2019deep,schwab2020big}, plug-and-play priors \cite{venkatakrishnan2013plug,chan2016plug,romano2017little}, deep image priors \cite{dittmer2020regularization,ulyanov2018deep}, variational networks \cite{kobler2017variational,hammernik2018learning}, network cascades \cite{schlemper2017deep,kofler2018u}, learned iterative schemes \cite{adler2017solving,aggarwal2018modl,chang2017one,hauptmann2020multi,kofler2021end,ADMMnet} and learned regularizers \cite{antholzer2020discretization,li2020nett,lunz2018adversarial,mukherjee2020learned}.

Classical deep learning approaches may lack data consistency for unknowns very different from the training data. To address this issue, in \cite{li2020nett} a deep learning approach named NETT (NETwork Tikhonov) regularization has  been introduced  which considers minimizers of  the NETT functional
\begin{equation}\label{eq:nett}
 \nett_{\alpha, \data^\delta}(\signal)  \coloneqq  \similarity(\Ko \signal, \data^\delta)    +  \alpha \pen( \encoder ( \signal ) ) \,.
\end{equation}
Here, $\similarity$ is a similarity measure, $ \encoder \colon \X \rightarrow \Xi$ is a trained neural network, $\pen \colon \Xi \rightarrow [0, \infty]$ a functional and $\alpha > 0$ a regularization parameter.  In \cite{li2020nett} it is shown that under suitable assumptions, NETT yields a convergent regularization method. This in particular includes provable stability guarantees and error estimates. Moreover, a training strategy has been proposed, where $\encoder$ is trained such that $\pen \circ \encoder$ favors artifact-free reconstructions over reconstructions with artifacts.

\subsection{The augmented NETT} \label{subsec:coerciveNETT}

One of the main assumptions for the analysis of \cite{li2020nett}  is the coercivity of the regularizer $\pen \circ \encoder$ which requires special care in network design and training. In order to overcome this limitation, we propose an augmented form of the  regularizer for which we are able to rigorously prove  coercivity.  More precisely, for fixed $c>0$, we consider minimizers $\signal_\alpha^\delta$ of the augmented NETT functional
\begin{equation} \label{eq:anett}
\anett_{\alpha, \data^\delta}(\signal)
 \coloneqq  \similarity(\Ko \signal, \data^\delta)
+ \alpha \left( \pen(\encoder(\signal)) + \frac{\aval}{2} \norm{\signal - (\decoder \circ \encoder)(\signal)}_2^2 \right).
\end{equation}
Here, $\similarity \colon \Y \times \Y \rightarrow [0, \infty]$ is a similarity measure and $\decoder \circ \encoder \colon \X \rightarrow \X$ is an encoder-decoder network trained such that for any signal $\signal$ on a signal manifold we have $(\decoder \circ \encoder)(\signal) \simeq \signal$ and that $\pen(\encoder(\signal))$ is small. We term this approach augmented NETT (aNETT) regularization.  In this work we provide a mathematical  convergence analysis for aNETT,  present a novel modular training strategy and investigate its practical performance.

The term $\pen(\encoder(\signal))$ implements learned prior knowledge on the encoder coefficients, while smallness of  $\norm{\signal - (\decoder \circ \encoder)(\signal)}_2^2$ forces $\signal$ to be close to the signal manifold. The latter term also  guarantees the coercivity of  \eqref{eq:anett}.  In the original NETT version \eqref{eq:nett}, coercivity of the regularizer  requires coercivity conditions on the network involved. Indeed, in the numerical experiments, the authors of \cite{li2020nett} observed a semi-convergence behaviour when minimizing \eqref{eq:nett}, so early stopping of the iterative minimization scheme has been used as additional regularization. We  attribute this semi-convergence behavior to the non-coercivity of the regularization term.  In the present, paper we address this issue systematically by augmentation of the NETT functional which guarantees coercivity and allows a more stable minimization. Coercivity is also one main ingredient for the mathematical convergence analysis.

An interesting practical instance of aNETT takes  $\pen$ as a weighted $\ell^q$-norm enforcing sparsity of the encoding coefficients \cite{daubechies2014sparsity,grasmair2008sparse}. An important example for the  similarity measure is given by the squared norm distance, which from  a statistical viewpoint can be motivated by a Gaussian white noise model. General similarity measures  allow us to adapt to different noise models which can be more appropriate for certain problems.

\subsection{Main contributions}

The contributions of this paper are threefold.  As described in more detail below, we introduce the aNETT framework, mathematically analyze its convergence, and propose a practical implementation that is applied to tomographic limited data problems.                  

\begin{itemize}[ topsep=0em, itemsep=0em]
\item The first  contribution is to introduce the structure of the aNETT regularizer $\reg (\signal) = \pen(\encoder(\signal)) +  (\aval / 2) \, \norm{\signal - (\decoder \circ \encoder)(\signal)}_2^2  $. The combination of the two terms in the regularizer, even in the case of a linear encoder seems to be new.  The term $\pen(\encoder(\signal))$   enforces regularity  of the analysis coefficients, which is an ingredient  in most of existing variational regularization techniques.  For example, this  includes sparse regularization in frames or dictionaries,  regularization with Sobolev norms or total variation regularization. On the other hand, the augmented term $\norm{\signal - (\decoder \circ \encoder) ( \signal)}_2^2$ penalized  distance  to the signal manifold. It is the combination of these two terms that results in a stable reconstruction scheme without the need of strong assumptions on the involved networks.

\item The second main contribution is the theoretical analysis of aNETT \eqref{eq:anett} in  the context of regularization theory.   We investigate the case where the image domain of the encoder is given by $\Xi = \ell^2(\Lambda)$ for some countable set $\Lambda$, and $\pen$ is a coercive functional measuring the complexity of the encoder  coefficients.  The presented  analysis is in the spirit of  the analysis of NETT given in \cite{li2020nett}. However,  opposed to NETT, the required coercivity property is derived naturally for the class of considered regularizers. This  supports the use of  the regularizer $\reg$ also from a theoretical side.  Moreover, the convergence rates results presented here uses assumptions significantly different  from \cite{li2020nett}.   While we present our analysis for the transform domain $\Xi = \ell^2(\Lambda)$ we could replace the encoder space by a general Hilbert or Banach space.

\item As a third main contribution we propose a modular  strategy for training $\decoder \circ \encoder$ together with a possible network architecture. First, independent  of the given inverse problem, we train a $\pen$-penalized autoencoder that  learns representing signals from the training data with low complexity. In the second step, we train a task-specific network which can be adapted to specific inverse  problem at hand. In our  numerical experiments, we empirically found this modular training strategy to be superior to directly adapting the  autoencoder to the inverse problem. For the $\pen$-penalized autoencoder, we train the  modified version described in \cite{obmann2021deep} of the  tight frame U-Net of \cite{han2018framing} in a way such that  $\pen$ poses additional constraints on the autoencoder during the training process.

\end{itemize}

\subsection{Outline}

In Section~\ref{sec:analysis} we present the mathematical convergence analysis of aNETT. In particular, as an auxiliary result, we establish the coercivity of the regularization term. Moreover, we  prove stability and derive convergence rates. Section~\ref{sec:realization}  presents practical aspects for aNETT.  We propose a possible architecture and training strategy for the networks, and a possible ADMM based scheme to obtain minimizers of the aNETT functional. In Section~\ref{sec:application},   we  present reconstruction results and compare aNETT with other deep learning based reconstruction methods. The paper concludes with a short summary and discussion.
Parts of this paper were presented at the ISBI 2020 conference and the corresponding proceedings~\cite{obmann2020sparse}. Opposed to the proceedings, this article treats a  general similarity measure $\similarity$ and considers a general complexity measure $\pen$. Further, all  proofs and all numerical results presented in this paper are new.

\section{Mathematical analysis} 
\label{sec:analysis}

In this section we prove the stability and convergence of aNETT as regularization method. Moreover, we derive convergence rates in the form of quantitative error estimates between exact solutions for noise-free data and aNETT regularized solutions for noisy data. 

\subsection{Assumptions and coercivity results}

For our convergence analysis we make use of the following assumptions on the underlying spaces and operators involved.
\begin{cond}[General assumptions for  aNETT] \label{cond:a} \hfill
\begin{enumerate}[label=(A\arabic*), leftmargin=3em, topsep=0em, itemsep=0em]
\item $\X$ and $\Y$ are Hilbert spaces.
\item $\Xi = \ell^2(\Lambda)$ for a countable $\Lambda$.
\item $\Ko \colon \X \rightarrow \Y$ is weakly sequentially continuous.
\item \label{a4} $\encoder \colon \X \rightarrow \Xi$ is weakly sequentially continuous.
\item \label{a5} $\decoder \colon \Xi \rightarrow \X$ is  weakly sequentially continuous.
\item  \label{a6} $\pen \colon \Xi \rightarrow [0, \infty]$ is coercive and weakly sequentially lower semi-continuous.
\end{enumerate}
\end{cond}
We set $\auto\coloneqq \decoder \circ \encoder$ and, for given $\aval > 0$, define
\begin{equation} \label{eq:areg}
	\reg \colon \X \to [0, \infty] \colon \signal \mapsto \pen(\encoder(\signal)) + \frac{\aval}{2} \norm{\signal - (\decoder \circ \encoder)(\signal)}_2^2 \,,
\end{equation}
which we refer to as the  aNETT (or augmented NETT) regularizer.
	
According to  \ref{a4}-\ref{a6}, the aNETT regularizer  is weakly sequentially  lower semi-continuous. As a main ingredient for our analysis we next prove its  coercivity.

\begin{theorem}[Coercivity of  the aNETT regularizer $\reg$] \label{thm:coercivity}
If Condition \ref{cond:a} holds, then  $\reg \colon \X \rightarrow [0, \infty]$ is coercive.
\end{theorem}

\begin{proof}
Let $(\signal_n)_{n\in \N}$ be some sequence in $\X$ such that $(\reg(\signal_n))_{n\in \N}$ is bounded. Then by definition of $\reg$ it follows that $(\pen(\encoder(\signal_n)))_{n\in \N}$ is bounded and by coercivity of $\pen$ we have that $(\encoder(\signal_n))_{n\in \N}$ is also bounded. By assumption, $\decoder$ is weakly sequentially continuous and thus $(\norm{\auto(\signal_n)})_{n\in \N}$ must be bounded.
Using that  $\aval > 0$, we obtain the inequality    $\norm{\signal_n}^2  \leq 2\norm{\signal_n - \auto(\signal_n)}^2 + 2\norm{\auto(\signal_n)}^2
 \leq (4 / \aval) \,  \reg(\signal_n) + 2 \norm{\auto(\signal_n)}^2$. This shows that $(\signal_n)_{n\in \N}$ is bounded and therefore that $\reg$ is coercive.
\end{proof}

\begin{example}[Sparse aNETT regularizer]
To obtain a sparsity promoting regularizer we can choose $\pen(\xi) = \lVert \xi  \rVert_{1; w}  \coloneqq \sum_{\lambda \in \Lambda} w_\lambda \abs{\xi_\lambda}^q$ where $q \in [1, 2]$ and $\inf_\lambda w_\lambda > 0$. Since $q \in [1,2]$ we have $\norm{\edot }^2 \leq  (\inf_\lambda w_\lambda)^{-1/q} \, \lVert \edot  \rVert_{1; w}$  and hence $ \lVert \edot  \rVert_{1; w}$ is coercive.
As a sum of weakly sequentially lower semi-continuous functionals it is also weakly sequentially lower semi-continuous \cite{grasmair2008sparse}.
Therefore Condition \ref{a6} is satisfied for the weighted  $\ell^q$-norm.
Together with Theorem~\ref{thm:coercivity}, we conclude that the resulting  weighted sparse aNETT regularizer $\signal \mapsto \lVert \encoder(\signal)  \rVert_{1; w}  + (\aval / 2 ) \,  \norm{\signal - (\decoder \circ \encoder)(\signal)}_2^2 $ is a coercive and weakly sequentially lower semi-continuous functional.
\end{example}

For the further analysis we will make the following assumptions including the similarity measure $\similarity \colon \Y \times \Y \rightarrow [0, \infty]$.

\begin{cond}[Similarity measure] \label{assumption:b} \mbox{}
\begin{enumerate}[label=(B\arabic*), leftmargin=3em, topsep=0em, itemsep=0em]
\item\label{b1} $\forall \data_0, \data_1 \colon \similarity(\data_0, \data_1) = 0 \Leftrightarrow \data_0 = \data_1$.
\item\label{b2} $\similarity$ is sequentially lower semi-continuous with respect to the weak topology in the first  and the  norm topology in the second argument.
\item\label{b3} $\forall (\data_n)_{n \in \N} \in \Y^\N \colon (\similarity(\data, \data_n) \rightarrow 0 \Rightarrow \data_n \rightarrow \data$ as $n \to \infty$).
\item\label{b4} \label{b5} $\similarity(\data, \data_n) \rightarrow 0$ as $n \to \infty$ $\Rightarrow$ ($\forall z \in \Y\colon \similarity(z, \data) < \infty \Rightarrow \similarity(z, \data_n) \rightarrow \similarity(z, \data)$).
\item\label{b5} $\forall  \data \in \Y \; \forall  \alpha > 0\colon$ $\exists \signal \in \X$ with $\similarity(\Ko \signal, \data) + \alpha \reg(\signal) < \infty$.
\end{enumerate}
\end{cond}

While \ref{b1}-\ref{b4} restrict the choice of the similarity measure, \ref{b5} is a technical assumption involving the  forward operator, the regularizer and the similarity measure, that is required for the existence of minimizers. For a more detailed discussion of these assumptions we refer to \cite{poschl2008tikhonov}.

\begin{example}[Similarity measures using the norm]
The classical example of a similarity measure satisfying   \ref{b1}-\ref{b4}  is given by $\similarity(\data_0, \data_1) = \norm{\data_0 - \data_1}^p$ for some $p \geq 1$ and more generally by $\similarity(\data_0, \data_1) = \psi(\norm{\data_0 - \data_1})$, where $\psi \colon [0, \infty) \rightarrow [0, \infty)$ is a continuous and monotonically increasing function that satisfies $\forall t \geq 0 \colon \psi(t) = 0 \Leftrightarrow t = 0$.
\end{example}

Taking into account Theorem \ref{thm:coercivity},  Conditions~\ref{cond:a} and \ref{assumption:b} imply that the aNETT functional $\anett_{\alpha, \data}$ defined by  \eqref{eq:anett}, \eqref{eq:areg}  is  proper, coercive and weakly sequentially  lower semi-continuous.  This in particular implies the existence of minimizers of  $\anett_{\alpha, \data}$ for all data $\data \in \Y$ and regularization parameters $\al >0$.

\subsection{Stability} \label{subsec:analysis}

Next we prove the stability of minimizing the aNETT functional $\anett_{\alpha, \data}$ regarding  perturbations of the data $\data$.

\begin{theorem}[Stability] \label{thm:stability}
Let Conditions~\ref{cond:a} and \ref{assumption:b} hold, $\data \in \Y$ and $\alpha > 0$. Moreover, let  $(\data_n)_{n\in \N} \in \Y^\N$ be a sequence of perturbed data with $\similarity(\data, \data_n) \rightarrow 0$ and consider minimizers $\signal_n \in \argmin \anett_{\alpha, \data_n}$. Then the sequence $(\signal_n)_{n\in \N}  \in \X^\N$ has at least one weak accumulation point and weak accumulation points are minimizers of $\anett_{\alpha, \data}$. Moreover, for any weak accumulation point $\signal^\ddagger$ of $(\signal_n)_{n\in \N} $  and any subsequence $(\signal_{\tau(n)})_{n \in \N}$  with $\signal_{\tau(n)} \rightharpoonup \signal^\ddagger$ we have $\reg (\signal_{\tau(n)}) \rightarrow \reg(\signal^\ddagger)$.
\end{theorem}

\begin{proof}
Let $\signal_\star \in \X$ be such that $\anett_{\alpha, \data}(\signal_\star) < \infty$. By definition of $\signal_n$ we have $
\anett_{\alpha, \data_n}(\signal_n) \leq \anett_{\alpha, \data_n}(\signal_\star)$.
Since by assumption $\similarity(\data, \data_n) \rightarrow 0$ and $\similarity(\Ko \signal_\star, \data) < \infty$, we have $\similarity(\Ko \signal_\star, \data_n) \rightarrow \similarity(\Ko \signal_\star, \data)$. This implies that $\anett_{\alpha, \data_n}(\signal_\star)$ is bounded by some positive constant $m_\star$ for sufficiently large  $n$. By definition of $\anett_{\alpha, \data_n}$ we have
$\alpha \reg(\signal_n) \leq \anett_{\alpha, \data_n}(\signal_n) \leq m_\star + \alpha \reg(\signal_\star)$. Since $\reg$ is coercive it follows that $(\signal_n)_{n\in \N}$ is a bounded sequence and hence it has a weakly convergent subsequence.

Let $(\signal_{\tau(n)})_{n \in \N}$ be a weakly convergent subsequence of $(\signal_n)_{n \in \N}$ and denote its limit by $\signal^\ddagger$. By the lower semi-continuity we get
$\similarity(\Ko \signal^\ddagger, \data) \leq \liminf_{n \to \infty} \similarity(\Ko \signal_{\tau(n)}, \data_{\tau(n)}) $ and  $\reg(\signal^\ddagger) \leq \liminf_{n \to \infty} \reg(\signal_{\tau(n)})$.
Thus for all  $\signal \in \X$ with $\anett_{\alpha, \data}(\signal) < \infty$ we have
\begin{multline*}
 \anett_{\alpha, \data}(\signal^\ddagger)
 \leq
 \liminf_{n \to \infty}  \similarity(\Ko \signal_{\tau(n)}, \data_{\tau(n)}) + \alpha \liminf_{n \to \infty}  \reg(\signal_{\tau(n)})
 \\ \leq
 \limsup_{n \to \infty} \similarity(\Ko \signal_{\tau(n)}, \data_{\tau(n)}) + \alpha \reg(\signal_{\tau(n)})
  \leq \limsup_{n \to \infty} \anett_{\alpha, \data_{\tau(n)}}(\signal)   = \anett_{\alpha, \data}(\signal)  \,.
\end{multline*}
This shows that $\signal^\ddagger  \in \argmin \anett_{\alpha, \data}$ and, by considering $\signal = \signal^\ddagger$ in the  above displayed equation, that $ \anett_{\alpha, \data_{\tau(n)}}(\signal_{\tau(n)}) \to  \anett_{\alpha, \data}( \signal^\ddagger )$. Moreover, we have
\begin{multline*}
\limsup_{n \to \infty} \alpha \reg(\signal_{\tau(n)})  \leq \limsup_{n \to \infty} \anett_{\alpha, \data_{\tau(n)}}(\signal_{\tau(n)}) - \liminf_{n \to \infty} \similarity(\Ko \signal_{\tau(n)}, \data_{\tau(n)})
 \\ \leq \anett_{\alpha,\data}(\signal^\ddagger) - \similarity(\Ko \signal^\ddagger, \data)
 = \alpha \reg(\signal^\ddagger )   \,.
 \end{multline*}
 This shows $\reg(\signal_{\tau(n)}) \to  \reg( \signal^\ddagger )$ as $n \to \infty$ and concludes the proof.
\end{proof}

In the following we say that the similarity measure $\similarity$ satisfies the quasi triangle-inequality if there is some $q \geq 1$ such that
\begin{equation} \label{eq:triangle}
\forall \data_0, \data_1, \data_2 \in \Y \colon \quad \similarity(\data_0, \data_1) \leq q \cdot \bigl(  \similarity(\data_0, \data_2) + \similarity(\data_2, \data_1) \bigr) \,.
\end{equation}
While this property is essential for deriving convergence rate results, we will show below that it is not enough to guarantee stability of minimizing the augmented NETT  functional in the sense of Theorem~\ref{thm:stability}.  Note  that \cite{li2020nett} assumes the  quasi triangle-inequality \eqref{eq:triangle} instead of Condition~\ref{b4}. The following remarks shows that \eqref{eq:triangle}  is not sufficient for the  stability result of Theorem~\ref{thm:stability} to hold and  therefore Condition~\ref{b4} has to be added to the list of assumptions in \cite{li2020nett} required for the stability.

\begin{example}[Instability in the absence of Condition~\ref{b4}]
Consider the similarity measure $\similarity \colon \Y \times \Y   \to [0, \infty]$ defined by
\begin{equation} \label{eq:counterexample}
\similarity(\data_0, \data_1) \coloneqq H( \data_1) \,  \norm{\data_0 - \data_1}^2 \,,
\end{equation}
where $H \colon \Y \to [0,1]$ is defined by $H( \data_1)  = 1$  if $\norm{\data_1} \leq 1$ and $H( \data_1)  = 2$  otherwise.
Moreover, choose $\X = \Y$, let $\Ko = \mathrm{Id}$ be the identity operator and suppose the  regularizer takes  the form $\reg  = \norm{\edot}^2$.

\begin{itemize}[wide,topsep=0em, itemsep=0em]
\item The similarity measure defined in \eqref{eq:counterexample} satisfies  \ref{b1}-\ref{b3}:  Convergence with respect to $\similarity$ is equivalent to convergence in norm which implies  that \ref{b3} is satisfied. Moreover, we have  $\forall \data_0, \data_1\colon  \similarity(\data_0, \data_1) = 0 \Leftrightarrow \data_0 = \data_1$, which is \ref{b1}. Consider sequences $z_n \to z$ and  $\data_n \rightharpoonup  \data$. The sequential lower semi-continuity stated in  \ref{b2} can be derived by separately looking at the cases $\norm{z} \leq 1$ and $\norm{z} > 1$. In the first case, by the  continuity of the norm we have  $\forall n \in \N \colon H(z) \leq H(z_n)$. In the second case, we have $\norm{\data_n} > 1$ for $n$ sufficiently large. In both cases, the  lower semi-continuity property follows from the weak lower semi-continuity property of the norm.

\item The similarity measure defined by  \eqref{eq:counterexample} does not satisfy  \ref{b4}: To see this, we define $\data_n \coloneqq (1 + \delta_n) \data$ where $\norm{\data} = 1$ and $\delta_n>0$ is taken as a non-increasing sequence converging to zero. We have $\data_n \to \data$ and hence also $\similarity(\data, \data_n) \to 0$ as $n \to \infty$. For any $z \in \Y$ we have  $\similarity(z, \data) < \infty$ and $\similarity(z, \data_n) = 2 \norm{z - \data_n}^2 \rightarrow 2 \norm{z - \data}^2 = 2  \similarity( z, \data)$ as $n \to \infty$.  In particular, $\similarity(z, \data_n)$ does  not converge to $\similarity( z, \data)$ if $z \neq \data$ and therefore \ref{b5} does not hold. In summary, all  requirement for Theorem \ref{thm:stability} are satisfied, except  of the continuity assumption \ref{b4}.

\item
We have  $ \norm{\data_0 - \data_1}^2 \leq \similarity(\data_0, \data_1) \leq 2\norm{\data_0 - \data_1}^2$ which implies that the similarity measure satisfies the quasi triangle-inequality \eqref{eq:triangle}.
However as shown next, this is not  sufficient for  stable reconstruction in the sense of Theorem~\ref{thm:stability}. To that end, let  $\data_n \coloneqq (1 + \delta_n) \data$ with $\norm{\data} = 1$ and $\delta_n \downarrow 0$ be  as above and let $\alpha > 0$. In particular, $\similarity( \signal , \data_n ) = 2 \norm{\signal - \data_n}^2$ and $\similarity( \signal , \data ) =  \norm{\signal - \data_n}^2$. Therefore the minimizer  of $ \argmin \similarity( \signal , \data_n ) + \alpha \norm{\signal}^2$ with perturbed data $\data_n$ is given by $ \signal_n =  \data_n/(1 + \alpha/2)$ and the minimizer of   $\argmin \similarity( \signal , \data_n ) + \alpha \norm{\signal}^2$ for data $\data$ is given by $\signal^\ddagger =  \data/(1 + \alpha)$. We see that $ \signal_n  \to \data / (1 + \alpha/2)$, which is clearly different from $\signal^\ddagger$.
In particular, minimizing $\similarity( \edot , \data ) + \alpha \norm{\edot}^2$  does not stably depend on data $\data$. Theorem  \ref{thm:stability} states that stability holds if  \ref{b4} is satisfied.
\end{itemize}
\end{example}

While the above example may seem  somehow constructed, it shows that one has to be careful when choosing the similarity measure in order to obtain a stable reconstruction scheme.

\subsection{Convergence} \label{subsec:convergence}

In this subsection  we prove that  minimizers of the aNETT functional   for noisy data  convergence to  $\reg$-minimizing solution of the equation $\Ko \signal = \data$  as the noise level goes to zero and the regularization parameter is chosen properly.  Here and below, we use the following notation.

\begin{definition}[$\reg$-minimizing solutions]
For $\data \in \Y$, we call an element $\signal^\ddagger \in \dom (\reg)$  an $\reg$-minimizing solution of the equation $\Ko \signal = \data$ if
\begin{equation*}
 \signal^\ddagger \in \argmin \{ \reg(\signal) \colon \signal \in \dom(\reg)  \wedge  \Ko \signal = \data\} \,.
\end{equation*}
\end{definition}
An $\reg$-minimizing solution always exists provided that  data  satisfies $\data \in \Ko (\dom (\reg))$, which means that the equation $\Ko \signal = \data$ has at least on solution with finite value of $\reg$.  To see this, consider a sequence of solutions $(\signal_n)_{n \in \N}$ with $\reg (\signal_n) \rightarrow \inf \{ \reg(\signal) \colon \Ko \signal = \data  \wedge \reg(\signal) < \infty \}$. Since $\reg$ is coercive there exists a weakly  convergent subsequence $(\signal_{\tau(n)})_{n \in \N}$ with weak limit $\signal^\ddagger$. Using the weak sequential  lower semi-continuity of $\reg$ one concludes  that $\signal^\ddagger$ is  an $\reg$-minimizing solution.

We first show weak convergence.

\begin{theorem}[Weak convergence of aNETT] \label{thm:weakconvergence}
Suppose  Conditions~\ref{cond:a} and \ref{assumption:b} are satisfied. Let $\data \in \Ko (\dom (\reg))$, $(\delta_n)_{n\in \N} \in (0, \infty)^\N$ with $\delta_n \rightarrow 0$ and let $(\data_n)_{n\in \N}\in \Y^\N$ satisfy $\similarity(\data, \data_n) \leq \delta_n$. Choose $\alpha_n>0$ such that $\lim_{n \to \infty} \delta_n / \alpha_n = \lim_{n \to \infty} \alpha_n = 0$ and let $\signal_n \in \argmin \anett_{\alpha_n, \data_n}$. Then the following hold:
\begin{enumerate}[label=(\alph*), topsep=0em, itemsep=0em]
\item\label{weak-a} $(\signal_n)_{n\in \N}$ has at least one weakly convergent subsequence.
\item\label{weak-b} All accumulation points of $(\signal_n)_{n\in \N}$ are  $\reg$-minimizing solutions of $\Ko \signal = \data$.
\item\label{weak-c} For every convergent subsequence $(\signal_{\tau(n)})_{n \in \N}$ it holds $\reg(\signal_{\tau(n)}) \rightarrow \reg(\signal^\ddagger)$.
\item\label{weak-d} If the $\reg$-minimizing solution $\signal^\ddagger$ is unique then $\signal_n \rightharpoonup \signal^\ddagger$.
\end{enumerate}
\end{theorem}

\begin{proof}
\ref{weak-a}: Because  $\data \in \Ko (\dom (\reg))$, there exists an  $\reg$-minimizing solution of the equation $\Ko \signal = \data$ which we denote by $\signal^\ddagger$. Because $\signal_n \in \argmin \anett_{\alpha_n, \data_n}$ we have
\begin{multline} \label{conv:aux}
\alpha_n \reg(\signal_n)  \leq \similarity(\Ko \signal_n, \data_n) + \alpha_n \reg(\signal_n)
 \leq \similarity(\Ko \signal^\ddagger, \data_n) + \alpha_n \reg(\signal^\ddagger) \\
 = \similarity(\data, \data_n) + \alpha_n \reg(\signal^\ddagger)
 \leq \delta_n + \alpha_n \reg(\signal^\ddagger) \,.
\end{multline}
Because  $\alpha_n, \delta_n \to 0$ this shows that $(\reg(\signal_n))_{n \in \N}$ is bounded. Due to the coercivity of the aNETT regularizer (see Theorem \ref{thm:coercivity}), this implies that $(\signal_n)_{n\in \N}$ has a weakly convergent subsequence.

\ref{weak-b}, \ref{weak-c}: Let $(\signal_{\tau(n)})_{n \in \N}$ be a weakly convergent subsequence of $(\signal_n)_{n\in \N}$  with limit  $\signal_\star$.  From the weak lower semi-continuity  we get $\similarity(\Ko \signal_\star, \data)  \leq \liminf_{n \to \infty} \similarity(\Ko \signal_{\tau(n)}, \data_{\tau(n)}) + \alpha_{\tau(n)} \reg(\signal_{\tau(n)})
 \leq \liminf_{n \to \infty} \similarity(\Ko \signal^\ddagger, \data_{\tau(n)}) + \alpha_{\tau(n)}  \reg(\signal^\ddagger)   = 0$,
which shows that $\signal_\star$ is a solution of $\Ko \signal = \data$. Moreover,
\begin{equation*}
 	\reg (\signal^\ddagger)  \leq \reg (\signal_\star) \leq \liminf_{n \to \infty} \reg(\signal_{\tau(n)})  \leq \liminf_{n \to \infty}  \frac{\delta_{\tau(n)}}{\alpha_{\tau(n)}} + \reg(\signal^\ddagger) = \reg(\signal^\ddagger)
\end{equation*}
where for the second last inequality we used \eqref{conv:aux} and for the last equality we used that $\delta_n/\alpha_n \to  0$. Therefore,   $\signal_\star$ is an $\reg$-minimizing solution of the equation $\Ko \signal = \data$. In a similar manner we derive $\reg (\signal_\star)  \leq \liminf_{n \to \infty} \reg (\signal_{\tau(n)}) \leq  \limsup_{n \to \infty}  \delta_{\tau(n)} / \alpha_{\tau(n)}  + \reg(\signal^\ddagger) = \reg (\signal_\star) $ which shows $\reg (\signal_{\tau(n)}) \rightarrow \reg (\signal_\star)$.

\ref{weak-d}: If $\Ko \signal = \data$ has a unique $\reg$-minimizing solution $\signal^\ddagger$, then every subsequence of $(\signal_n)_{n\in \N}$ has itself a subsequence weakly converging to $\signal^\ddagger$, which implies that $(\signal_n)_{n\in \N}$ weakly converges to the $\reg$-minimizing solution.
\end{proof}

Next we derive strong convergence of the regularized solutions. To this end we recall the absolute Bregman distance, the modulus of total nonlinearity and the total nonlinearity, defined in \cite{li2020nett}.

\begin{definition}[Absolute Bregman distance]
Let $\mathcal{F} \colon \X \rightarrow [0, \infty]$ be Gâteaux differentiable at $\signal_\star \in \X$. The absolute Bregman distance $\Delta_\mathcal{F}(\edot,  \signal_\star) \colon \X \rightarrow [0, \infty]$ at $\signal_\star$ with respect to $\mathcal{F}$ is defined by
\begin{align*}
\forall \signal \in \X  \colon \quad \Delta_\mathcal{F}(\signal, \signal_\star) \coloneqq \bigl|  \mathcal{F}(\signal) - \mathcal{F}(\signal_\star) - \mathcal{F}'(\signal)(\signal - \signal_\star)  \bigr| \,.
\end{align*}
Here and below $\mathcal{F}'(\signal_\star)$ denotes the Gâteaux derivative of $\mathcal{F}$ at $\signal_\star$.
\end{definition}

\begin{definition}[Modulus of total nonlinearity and total nonlinearity]
Let $\mathcal{F} \colon \X \rightarrow  [0, \infty]$ be Gâteaux differentiable at $\signal_\star \in \X$. We define the modulus of total nonlinearity of $\mathcal{F}$ at $\signal_\star$ as $\nu_\mathcal{F}(\signal_\star, \edot) \colon [0, \infty) \rightarrow [0, \infty) \colon t \mapsto  \inf \{ \Delta_\mathcal{F}(\signal, \signal_\star) \colon \norm{\signal - \signal_\star} = t \}$. We call $\mathcal{F}$ totally nonlinear at $\signal_\star$ if $\nu_\mathcal{F}(\signal_\star, t) > 0$ for all $t \in (0, \infty)$.
\end{definition}

Using these definitions we get the following convergence result in the norm topology.

\begin{theorem}[Strong convergence  of aNETT] \label{thm:strong}
Let Conditions~\ref{cond:a} and \ref{assumption:b} hold, $\data \in \Ko (\dom (\reg))$ and let $\reg$ be totally nonlinear at all $\reg$-minimizing solutions of $\Ko \signal = \data$. Let $(\data_n)_{n\in \N}, (\signal_n)_{n\in \N}, (\alpha_n)_{n\in \N}$ be  as in Theorem~\ref{thm:weakconvergence}. Then there is a subsequence $(\signal_{\tau(n)})_{n \in \N}$ which converges in norm to an $\reg$-minimizing solution $\signal^\ddagger$ of $\Ko \signal = \data$. If the $\reg$-minimizing solution is unique, then $\signal_n \to \signal^\ddagger$ as $n \to \infty$.
\end{theorem}

\begin{proof}
In \cite[Proposition 2.9]{li2020nett} it is shown that the total nonlinearity of $\reg$ implies that for every bounded sequence $(z_n)_{n \in \N}$ with  $\Delta_{\reg}(z_n, z) \rightarrow 0$ it holds that $z_n  \to z$.  Theorem~\ref{thm:weakconvergence} gives us a weakly converging subsequence $(\signal_{\tau(n)})_{n \in \N}$ of $ (\signal_n)_{n \in \N}$  with weak limit  $\signal^\ddagger$ and $\reg(\signal_{\tau(n)}) \to \reg(\signal^\ddagger)$.  By the definition of the absolute Bregman distance it follows that $\Delta_\reg(\signal_{\tau(n)}, \signal^\ddagger) \rightarrow 0$ and hence, together with \cite[Proposition 2.9]{li2020nett}, that  $\signal_{\tau(n)} \rightarrow \signal^\ddagger$. If the $\reg$-minimizing solution of  $\Ko \signal = \data$ is unique, then every subsequence has a subsequence converging to $\signal^\ddagger$  and hence the claim follows.
\end{proof}

\subsection{Convergence rates} \label{ssec:rates}

We will now prove convergence rates by deriving quantitative estimates for the absolute Bregman distance between  $\reg$-minimizing solutions for exact data and regularized solutions for noisy data.
The convergence rates will be derived under the additional assumption that $\similarity$ satisfies the quasi triangle-inequality \eqref{eq:triangle}.

\begin{proposition}[Convergence rates for aNETT] \label{prop:rate}
Let the assumptions of Theorem~\ref{thm:strong} be satisfied and suppose  that $\similarity$ satisfies the quasi triangle-inequality \eqref{eq:triangle} for some $q \geq 1$. Let $\signal^\ddagger \in \X$ be an $\reg$-minimizing solution of $\Ko \signal = \data$ such that $\reg$ is  Gâteaux differentiable at $\signal^\ddagger$ and assume there exist $\epsilon, c > 0$ with
\begin{equation}\label{eq:rate}
\forall \signal \in \X \colon  \quad \norm{\signal - \signal^\ddagger} \leq \epsilon \Rightarrow \Delta_{\reg}(\signal, \signal^\ddagger)
\leq \reg(\signal) - \reg(\signal^\ddagger) + c \sqrt{ \similarity(\Ko \signal, \Ko \signal^\ddagger) }\,.
\end{equation}
For any $\delta >0$, let $\data^\delta \in \Y$ be noisy data satisfying $\similarity(\data^\delta, \Ko \signal^\ddagger)^{1/2} \leq \delta$, $ \similarity(\Ko \signal^\ddagger, \data^\delta)^{1/2} \leq \delta$, and write $\signal_\alpha^\delta  \in  \argmin \anett_{\alpha, \data^\delta }$. Then the following hold:
\begin{enumerate}[label=(\alph*), topsep=0em, itemsep=0em]
\item \label{est-a} For sufficiently small $\alpha$, it holds $\Delta_{\reg}(\signal_\alpha^\delta, \signal^\ddagger)   \leq  \delta^2/\alpha - c \delta  \sqrt{q} + c^2 \alpha q /4$.
\item\label{est-b} If $\alpha \asymp \delta$, then  $\Delta_{\reg}(\signal_\alpha^\delta, \signal^\ddagger) = \mathcal{O} \bigl( \delta \bigr)$ as $\delta \to 0$.
\end{enumerate}
\end{proposition}

\begin{proof}
By definition of $\signal_\alpha^\delta$ we have $\similarity(\Ko \signal_\alpha^\delta, \data^\delta) + \alpha \reg(\signal_\alpha^\delta) - \alpha \reg(\signal^\ddagger) \leq \similarity(\Ko \signal^\ddagger, \data^\delta) \leq \delta^2$. By Theorem~\ref{thm:strong} for sufficiently small $\alpha$ we can assume that $\norm{\signal_\alpha^\delta - \signal^\ddagger} \leq \epsilon$   and hence
\begin{align*}
\alpha \Delta_\reg(\signal_\alpha^\delta, \signal^\ddagger)  & \leq \alpha \reg(\signal_\alpha^\delta) - \alpha \reg(\signal^\ddagger) + c \alpha \sqrt{\similarity(\Ko \signal_\alpha^\delta, \Ko \signal^\ddagger)  } \\
&=\reg(\signal_\alpha^\delta)  - \similarity(\Ko \signal_\alpha^\delta, \data^\delta) - \Bigl(  \reg(\signal_\alpha^\delta)  - \similarity(\Ko \signal_\alpha^\delta, \data^\delta) \Bigr)+ c \alpha\sqrt{\similarity(\Ko \signal_\alpha^\delta, \Ko \signal^\ddagger)  }
\\
& \leq \delta^2   - \similarity(\Ko \signal_\alpha^\delta, \data^\delta)  +  c \alpha  \delta \sqrt{ q } +  c \alpha \sqrt{ q  \similarity(\Ko \signal_\alpha^\delta, \data^\delta)} \,.
\end{align*}
Together with the  inequality of arithmetic and geometric means $(a + b) / 2 \geq \sqrt{ab}$  for  $ a =   \similarity(\Ko \signal_\alpha^\delta, \data^\delta)$ and $b = c^2 \alpha^2 q /4$ this implies
$\alpha \Delta_\reg(\signal_\alpha^\delta, \signal^\ddagger)  \leq \delta ^2 - c \alpha   \delta \sqrt{q} + c^2 \alpha^2 q /4$ which shows  \ref{est-a}.  Item \ref{est-b} is an immediate consequence of \ref{est-a}.
\end{proof}

The following results is our main  convergence rates result. It is  similar to Proposition \cite[Theorem~3.1]{li2020nett}, but uses different assumptions.

\begin{theorem}[Convergence rates for finite rank operators]
Let the assumptions of Theorem~\ref{thm:strong} be satisfied, take $\similarity (\data_1, \data_2) = \norm{\data_1 - \data_2}^2$, assume  $\Ko$ has finite dimensional range and that $\reg$ is Lipschitz continuous and Gâteaux differentiable.    For any $\delta >0$, let $\data^\delta \in \Y$ be noisy data satisfying $\norm{\data^\delta- \Ko \signal^\ddagger} \leq \delta$ and write $\signal_\alpha^\delta  \in  \argmin \anett_{\alpha, \data^\delta }$.
Then for the parameter choice  $\alpha \asymp \delta$ we have  the convergence rates result $\Delta_{\reg}(\signal_\alpha^\delta, \signal^\ddagger) = \mathcal{O} \bigl( \delta \bigr)$ as $\delta \to 0$.
\end{theorem}

\begin{proof}
According to Proposition \ref{prop:rate}, it is sufficient to show that \eqref{eq:rate} holds with $ \norm{\Ko \signal -\Ko \signal^\ddagger}$ in place of  $\similarity(\Ko \signal, \Ko \signal^\ddagger)^{1/2} $. 
For that purpose, let  $\mathbf{P}$ denote the orthogonal projection onto the null-space $\ker(\Ko)$ and let $L$ be a Lipschitz constant of $\reg$. Since $\Ko$ restricted to $\ker(\Ko)^\bot$ is injective with finite dimensional range, we can choose a constant $a >0$ such that  $\forall z \in \ker(\Ko)^\bot \colon \norm{\Ko z}  \geq a \norm{z}$.  

We first show the estimates 
\begin{align} \label{eq:linh1}
&	\forall \signal \in \X \colon \quad \reg(\signal^\ddagger)  - \reg(\signal) \leq (L/ a)   \, \norm{\Ko \signal - \Ko \signal^\ddagger}  
\\ \label{eq:linh2}
&	\forall \signal \in \X \colon \quad \abs{ \innerprod{\reg'(\signal^\ddagger), \signal - \signal^\ddagger} } \leq   ( \norm{\reg'(\signal^\ddagger)} / a) \,  \norm{\Ko \signal- \Ko \signal^\ddagger}
	\,.
\end{align} 
To that end, let  $\signal \in \X$ and write $\signal_0 \coloneqq (\signal^\ddagger - \mathbf{P} \signal^\ddagger) + \mathbf{P} \signal$. Then $\Ko \signal_0 = \Ko \signal^\ddagger$. Since $\signal^\ddagger$ is an $\reg$-minimizing solution, we have $ \reg(\signal^\ddagger) - \reg(\signal)  \leq  \reg(\signal_0) - \reg(\signal) \leq L \norm{\signal_0 - \signal}$.  Since $\signal_0 - \signal \in \ker(\Ko)^\bot$, we have  $\norm{\Ko \signal^\ddagger - \Ko \signal} = \norm{\Ko (\signal_0 -  \signal)} \geq a \norm{\signal_0 - \signal}$. The last two estimates prove  \eqref{eq:linh1}. Because $\signal^\ddagger$ is an $\reg$-minimizing solution, we have $  \innerprod{\reg'(\signal^\ddagger), \signal^\ddagger - \signal} =  0$ whenever $\signal^\ddagger - \signal \in \ker(\Ko)$. On the other hand, using that $\reg$ is Gâteaux differentiable and that  $\Ko$ has finite rank, shows $\vert \innerprod{\reg'(\signal^\ddagger), \signal^\ddagger - \signal} \rvert \leq \norm{\reg'(\signal^\ddagger)} \, a^{-1} \norm{\Ko (\signal^\ddagger -  \signal)}$ for $\signal^\ddagger - \signal \in \ker(\Ko)^\bot$. This proves  \eqref{eq:linh2}. 

Inequality \eqref{eq:linh1} implies $\abs{\reg(\signal) - \reg(\signal^\ddagger)} \leq  \reg(\signal) - \reg(\signal^\ddagger)  + 2 \, (L /  a)   \, \norm{\Ko \signal - \Ko \signal^\ddagger} $. Together with \eqref{eq:linh2} this yields
\begin{align*}
	 \Delta_\reg(\signal, \signal^\ddagger) 
	 &= \bigl\lvert \reg(\signal) - \reg(\signal^\ddagger)  - \innerprod{\reg'(\signal^\ddagger), \signal - \signal^\ddagger} \bigr\rvert 
	 \\&
	 \leq \abs{\reg(\signal) - \reg(\signal^\ddagger)} + \abs{\innerprod{\reg'(\signal^\ddagger), \signal^\ddagger - \signal}} 
	  \\&\leq \reg(\signal) - \reg(\signal^\ddagger) + (2 L +\norm{\reg'(\signal^\ddagger)} ) \, a^{-1} \,  \norm{\Ko \signal^\ddagger - \Ko \signal} \,,
\end{align*}
which proves  \eqref{eq:rate}  with  $c = (2 L +\norm{\reg'(\signal^\ddagger)} ) / a$.
\end{proof}

Note that the theoretical results stated remain valid, if we replace $\reg$ by a general coercive and weakly lower semi-continuous regularizer $\reg \colon \X \rightarrow [0, \infty]$.

\section{Practical realization} 
\label{sec:realization}

In this section we investigate practical aspects of aNETT. We present a possible network architecture together with a possible  training strategy in the  discrete setting. Further we discuss minimization of aNETT using the ADMM algorithm.

For the sake of clarity we restrict our discussion to the finite dimensional case where $\X = \R^{N \times N}$ and $\ell^2(\Lambda) = \R^\Lambda$ for  a finite index set $\Lambda$.

\subsection{Proposed modular aNETT training}

To find a suitable network $\decoder \circ \encoder$ defining  the aNETT regularizer $\reg (\signal) = \pen(\encoder(\signal)) + (\aval/2) \, \norm{ \signal - \decoder \circ \encoder (\signal)}^2$, we  propose a modular data driven approach that comes in two separate steps. In a first step, we train a $\pen$-regularized denoising autoencoder $\decoder^\pen \circ \encoder$ independent of the forward problem $\Ko$, whose purpose is to  well represent elements of a training data set by low complexity encoder coefficients.  In a second step, we train a task-specific network   that increases the ability of the  aNETT regularizer to  distinguish between clean images and images containing problem specific artifacts.

Let   $\signal_1, \dots, \signal_m  \in \M  $ denote the given set of artifact-free training phantoms.

\begin{itemize}[wide, topsep=0em, itemsep=0em]
\item\textsc{$\pen$-regularized autoencoder:}

First, an  autoencoder $\decoder^\pen  \circ \encoder$ is trained such that  $\signal_i$ is close to  $\decoder^\pen \circ \encoder ( \signal_i) $ and  that $\pen(\encoder(\signal_i))$ is small for the given training signals.
For that purpose, let  $(\decoder_\theta  \circ \encoder_\theta)_{\theta \in \Theta}$ be a  family of autoencoder networks, where $\encoder_\theta \colon \R^{N \times N} \to \R^\Lambda$ are encoder and $\decoder_\theta \colon \R^\Lambda \to \R^{N \times N}$  decoder networks, respectively.

To achieve that unperturbed images are sparsely represented by $\encoder$, whereas disrupted images are not, we apply the following training strategy. We randomly generate images   $\signal_i + a_i  \epsilon_i $  where  $\epsilon_i$  is additive Gaussian white noise with a standard deviation proportional to the mean value of $\signal_i$, and $a_i \in \{ 0,1 \}$ is a binary random variable that takes each value with probability $0.5$. For the numerical results below we use a standard deviation of 0.05 times the mean value of $\signal_i$. To select the particular autoencoder based on the training data, we  consider the following  training strategy 
\begin{equation}\label{eq:AEtrain}
    \theta^*  \in \argmin_\theta \frac{1}{m} \sum_{i=1}^m \norm{(\decoder_\theta \circ \encoder_\theta) (\signal_i + a_i \epsilon_i) -  \signal_i}_2^2  + \nu (1-a_i) \pen(\encoder_\theta(\signal_i))  + \beta  \norm{\theta}_2^2 \,,
\end{equation}
and set  $[\decoder^\pen, \encoder] \coloneqq [\decoder_{\theta^*}, \encoder_{\theta^*}]$. Here $\nu, \beta  > 0$  are regularization parameters.

Including perturbed signals $\signal_i + \epsilon_i$ in  \eqref{eq:AEtrain}  increases  robustness of the $\pen$-regularized autoencoder.  To enforce regularity for the encoder coefficients  only on the noise-free images, the penalty $ \pen(\encoder(  \edot )) $ is  only used for the noise-free inputs, reflected by the pre-factor $1-a_i$. Using auto-encoders, regularity for a signal class could also be achieved by means of dimensionality reduction techniques, where $\R^\Lambda$ is used as a bottleneck in the network architecture. However, in order to get a regularizer that is able to distinguish between perturbed and undisturbed perturbed signals we use $\R^\Lambda$ to be of sufficiently high dimensionality.

\item\textsc{Task-specific network:}

Numerical simulations showed that the  $\pen$-regularized autoencoder alone was not able to sufficiently well distinguish  between artifact-free training phantoms and  images containing problem specific artifacts.
In order to address this issue,  we compose the operator independent network with another network $\Uo$, that is  trained  to distinguish between image with and without problem specific artifacts.

For that purpose, we consider  randomly generated images  $z_1, \dots, z_m$ where either  $z_i = (\decoder^\pen \circ \encoder)  (\signal_i)$ or $z_i =  (\decoder^\pen \circ \encoder)( \Ko^{\ddagger} ( \Ko \signal_i + \eta_i))$  with equal probability. Here $\Ko^{\ddagger}$ is an approximate right inverse and $\eta_i$ are  error terms.
We choose a network architecture $( \Uo_\theta)_{\theta \in \Theta}$ and select $\Uo = \Uo_{\theta^*}$,  where
\begin{equation}\label{eq:U}
\theta^* \in \argmin \frac{1}{m}\sum_{i=1}^{m} \norm{\Uo_\theta   (z_i) - \signal_i}_2^2  +  \gamma \norm{\theta}_2^2 \,,
\end{equation}
for some regularization  parameter $\gamma > 0$.
In particular, the image  residuals $(\decoder^\pen \circ \encoder) (\Ko^{\ddagger} ( \Ko \signal_i + \eta _i)) - (\decoder^\pen \circ \encoder)( \signal_i)$ now depend on the specific inverse problem and we can consider them to consist of operator and training signal   specific artifacts.

The above training procedure  ensures that the network $\Uo$ adapts to the inverse problem at hand as well as to the $\pen$-regularized autoencoder.  Training the network $\Uo$ independently of $\decoder^\pen \circ \encoder$, or directly training the auto-encoder  to distinguish between images with and without problem specific artifacts,  we empirically found to perform considerably  worse.
\end{itemize}

The final autoencoder is then given as $\auto = \decoder \circ \encoder$ with modular decoder  $\decoder \coloneqq \Uo \circ \decoder^\pen$.   For the numerical results we take $( \Uo_\theta)_{\theta \in \Theta}$ as the  tight frame U-Net of \cite{han2018framing}. Moreover, we choose $(\decoder_\theta  \circ \encoder_\theta)_{\theta \in \Theta}$ as modified tight frame U-Net proposed in \cite{obmann2021deep} for deep synthesis regularization. In particular, opposed to the original tight frame U-net, the modified tight frame U-Net does not involve skip connections.

\subsection{Possible aNETT minimization}

For minimizing the aNETT functional  \eqref{eq:anett} we use the  alternating direction method of multiplies (ADMM) with scaled dual variable \cite{boyd2011distributed,gabay1976dual,glowinski75sur}.  For that purpose, the aNETT minimization problem is rewritten as the following constraint minimization problem
\begin{equation*}
 \left\{
   \begin{aligned}
    &\operatorname{arg\,min}_{\signal, \xi} && \similarity(\Ko \signal, \data^\delta) + \alpha \pen(\xi) + \alpha  \aval /2 \,  \norm{\signal - \auto(\signal)}^2
    \\
    &\operatorname{s.t.} && \encoder(\signal) = \xi \,.
    \end{aligned}
\right.
\end{equation*}
The resulting ADMM update scheme with  scaling parameter $\rho>0$ initialized by $\xi_0 = \encoder(\Ko^\ddagger \data^\delta )$ and $\eta_0 = 0$ then reads as follows:   

\begin{framed}
\begin{enumerate}[label=(S\arabic*), topsep=0.5em, itemsep=0.5em,leftmargin=3em]
\item\label{admm1} $ \signal_{k+1} = \argmin_\signal \similarity(\Ko \signal, \data^\delta) + (\alpha \aval / 2 ) \, \norm{\signal -\auto (\signal)}_2^2  +(\rho/2) \, \norm{\encoder(\signal) - \xi_k +  \eta_k }_2^2 $.
\item  \label{admm2} $\xi_{k+1} = \argmin_\xi \alpha \pen(\xi) + (\rho/2) \, \norm{\encoder( \signal_{k+1}  ) - \xi + \eta_k }_2^2$.
\item\label{admm3}   $\eta_{k+1} = \eta_{k+1} + \bigl( \encoder( \signal_{k+1}) - \xi_{k+1} \bigr) $.
\end{enumerate}
\end{framed}

One interesting feature of the above approach is that the signal update  \ref{admm1}  is independent of the possibly non-smooth penalty $\pen$. Moreover, the encoder update \ref{admm2} uses the proximal mapping of $\pen$ which in important special cases can be evaluated explicitly and therefore fast and exact. Moreover, it guarantees regular encoder coefficients during each iteration. For example, if we choose the penalty as the $\ell^1$-norm, then \ref{admm2} is a soft-thresholding step which results in sparse encoder coefficients. 
Step \ref{admm1} in typical cases has to be computed iteratively via an inner iteration. To find an approximate solution for \ref{admm1} for the results presented below we use gradient descent with at most $10$ iterations. We stop the gradient descent updates early if the difference of the functional evaluated at two consecutive iterations is below our predefined tolerance of $10^{-5}$. 

The concrete implementation  of the aNETT minimization   requires specification of the similarity measure, the total number of outer  iterations $N_{\mathrm{iter}}$, the step-size $\gamma$ for the iteration in \ref{admm1} and the parameters defining the aNETT functional. These specifications are selected dependent  of the inverse problem at hand. Table \ref{tab:parameter} lists the particular choices for the reconstruction scenarios considered in the following section.

\begin{table}[htb]
\centering
\begin{tabular}{| l | l | l | l | l| l | l | l |}
\toprule
 & $\alpha$ & $\aval$ & $\gamma$ & $N_{\mathrm{iter}}$ & $N_\varphi$ & noise  model & $\similarity$  \\
\midrule
Sparse view & $10^{-4}$ & $10^2$ & $5 \cdot 10^{-1}$ & 50 & $40$ & Gaussian, $\sigma = 0.02$ & $\similarity_{\norm{\edot}^2}$ \\
Low dose & $5 \cdot 10^{-3}$ & $10^2$ & $10^{-3}$ & $20$ & 1138  & Poisson, $p = 10^4$ & $\similarity_{\rm KL} $\\
Universality & $10^{-4}$ & $10^2$ & $5 \cdot 10^{-1}$ &$50$ & 160  & Gaussian, $\sigma = 0.02$ & $\similarity_{\norm{\edot}^2}$  \\
\bottomrule
\end{tabular}
\caption{Parameter specifications for proposed aNETT functional and it numerical minimization.}
\label{tab:parameter}
\end{table}

The ADMM scheme for  aNETT minimization shares similarities with existing iterative neural network based reconstruction methods. In particular, ADMM inspired plug-and-play priors \cite{venkatakrishnan2013plug, chan2016plug, romano2017little} may be most closely related. However, opposed the plug and play approach we can deduce convergence from existing results for ADMM for non-convex problems  \cite{wang2019global}. While  convergence of \ref{admm1}-\ref{admm3} and relations with plug and play priors are  interesting  and relevant, they are  beyond the scope of this work. This also applies to the comparison with other iterative minimization schemes for minimizing aNETT.

\begin{figure}[htb!]
\centering
 \includegraphics[width=\textwidth]{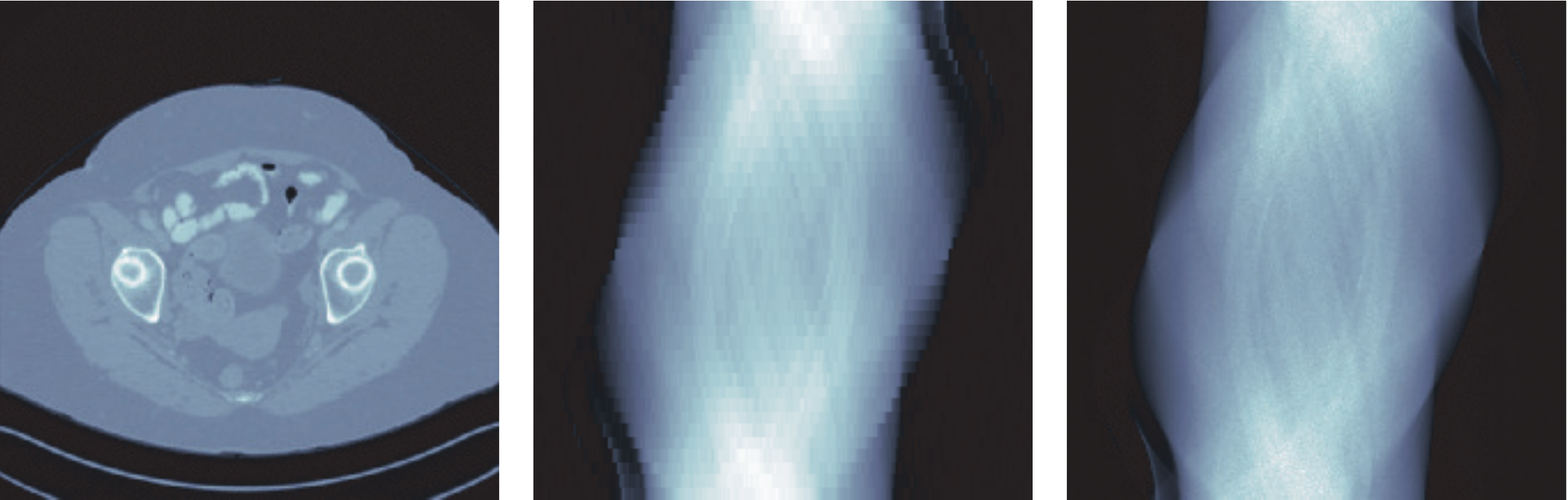}
  \caption{Left: Example image randomly drawn from the dataset. Middle:  Corresponding sparse view sinogram (40 directions). Right: low dose sinogram.}
  \label{fig:example}
\end{figure}

\section{Application to sparse view and low dose CT} \label{sec:application}

In this section we apply aNETT regularization to sparse view and low-dose computed tomography (CT). For the experiments we always choose $\pen$ to be the $\ell^1$-norm. The parameter specifications for the proposed aNETT functional and its numerical minimization are given in Table \ref{tab:parameter}. For quantitative evaluation, we use the peak-signal-to-noise-ratio (PSNR) defined  by
\begin{equation*}
    \operatorname{PSNR}(\signal, \signal_{\rm rec})  \coloneqq 20 \log_{10} \left( \frac{\max \signal}{\norm{\signal - \signal_{\rm rec}}_2} \right) \,.
\end{equation*}
Here $\signal \in \R^{N\times N}$  is the ground truth image and $\signal_{\rm rec}\in \R^{N\times N}$ its numerical reconstruction. Higher value of PSNR indicates better reconstruction.

\begin{table}[htb!]
\centering
\begin{tabular}{| l | l |  l | l | l | }
 \toprule
 PSNR & FBP & LPD   & Post & aNETT\\
\midrule
Sparse view & $23.8 \pm 1.3$ & $\mathbf{37.9 \pm 1.2}$   & $37.1 \pm 0.9$   & $37.1 \pm 1.0$   \\
\midrule
Low dose & $36.9 \pm 1.6$  & $43.6 \pm 1.3$  &   $\mathbf{44.1 \pm 1.5}$  & $43.9 \pm 1.3$\\
\midrule
Universality  & $32.4 \pm 1.6$  & \texttt{na} &  $37.7 \pm 0.8$ & $\mathbf{38.3 \pm 1.0}$\\
\bottomrule
\end{tabular}
\caption{Overview of metric results evaluated on the test-set. The values shown are the average of the PSNR $\pm$ the standard deviation calculated over the test dataset. The values in bold show the best results. The \texttt{na} entry means that LPD cannot be applied after changing the angular sampling pattern.} 
\label{tab:results}
\end{table}

\subsection{Discretization and dataset}

For  sparse view CT as well as for low dose CT we work with a discretization of the Radon transform 
$ (\Ro f)( \varphi, s  )  \coloneqq  \int_{\R} f( s \cos(\varphi) - t \sin(\varphi) , s\sin(\varphi)  +  t \cos(\varphi)  ) \mathrm{d} t$. The values  $(\Ro f)(\varphi, s )$ are integrals of the function  $f \colon \R^2 \to \R$  over lines  orthogonal to $(\cos(\varphi), \sin(\varphi))^\intercal$  for angle $\varphi \in [0, \pi)$ and signed distance $s \in \R$.    We discretize the Radon transform using the ODL library \cite{adler2017odl} where we assume that the function has compact support in $[-1,1]^2$ and sampled on an equidistant grid. We use $N_\varphi$ equidistant samples of  $\varphi \in [0, \pi)$ and $N_s$ equidistant samples of $s \in [-1.5, 1.5]$. In both cases, we end up with an inverse problem of the form  \eqref{eq:ip}, where $\Ko \colon \R^{N \times N} \to \R^{N_\varphi \times N_s} $ is the discretized linear forward operator. Elements $\signal  \in  \R^{N \times N}$ will be referred to as CT images and the elements  $\data \in \R^{N_\varphi \times N_s}$ as sinograms.

 or all results presented below we work  with image size $512 \times 512$  and use $N_s  =  768$. The number of angular samples $N_\varphi$ is  taken $40$ for low the dose CT and $N_\varphi = 1138$ for the low dose example. In both cases we use the CT images from the Low Dose CT Grand Challenge dataset \cite{mccollough2016tu} provided by the Mayo Clinic. The dataset consists of $512 \times 512$ grayscale images of $10$ different patients, where for  each patient there are multiple CT scanning series available. We use the split $7/2/1$ for training, validation and testing which corresponds to $4267/1143/526$ CT images in the respective sets.  We use the validation set to select networks which achieve the minimal loss on the validation set. The test set  is used to evaluate the final performance.  Note that by splitting the dataset according to patient we avoid validation and testing on images patients that  have already be seen during training time. An example image and the corresponding simulated  sparse view and low-dose data are shown in Figure~\ref{fig:example}.

\subsection{Numerical results}
\label{ssec:results}

We compare results of aNETT to the learned primal-dual algorithm (LPD) \cite{adler2018learned}, the tight frame U-Net  \cite{han2018framing} applied as post-processing network (CNN), and the filtered back-projection (FBP) as a base-line.  Minimization of the loss-function for all methods was done using Adam \cite{kingma2014adam} for $100$ epochs, cosine decay learning rate $\eta_t = (\eta_0/2) \cdot ( 1 + \cos (\pi  t/100) )$ with $\eta_0 = 10^{-3}$ in the $t$-th epoch, and a batch-size of $4$. For the LPD we take hyper-parameters $N_{\mathrm{primal}} = N_{\mathrm{dual}} = 5$ and $N = 7$ network iterations and train according to \cite{adler2018learned}.  For training of the tight frame U-Net we do not follow the patch approach of \cite{han2018framing} but instead use full images obtained with FBP as CNN inputs.
Training of all the networks was done on a GTX 1080 Ti with an Intel Xeon Bronze 3104 CPU.

\begin{figure*}[htb!]
\centering
\includegraphics[width=\textwidth]{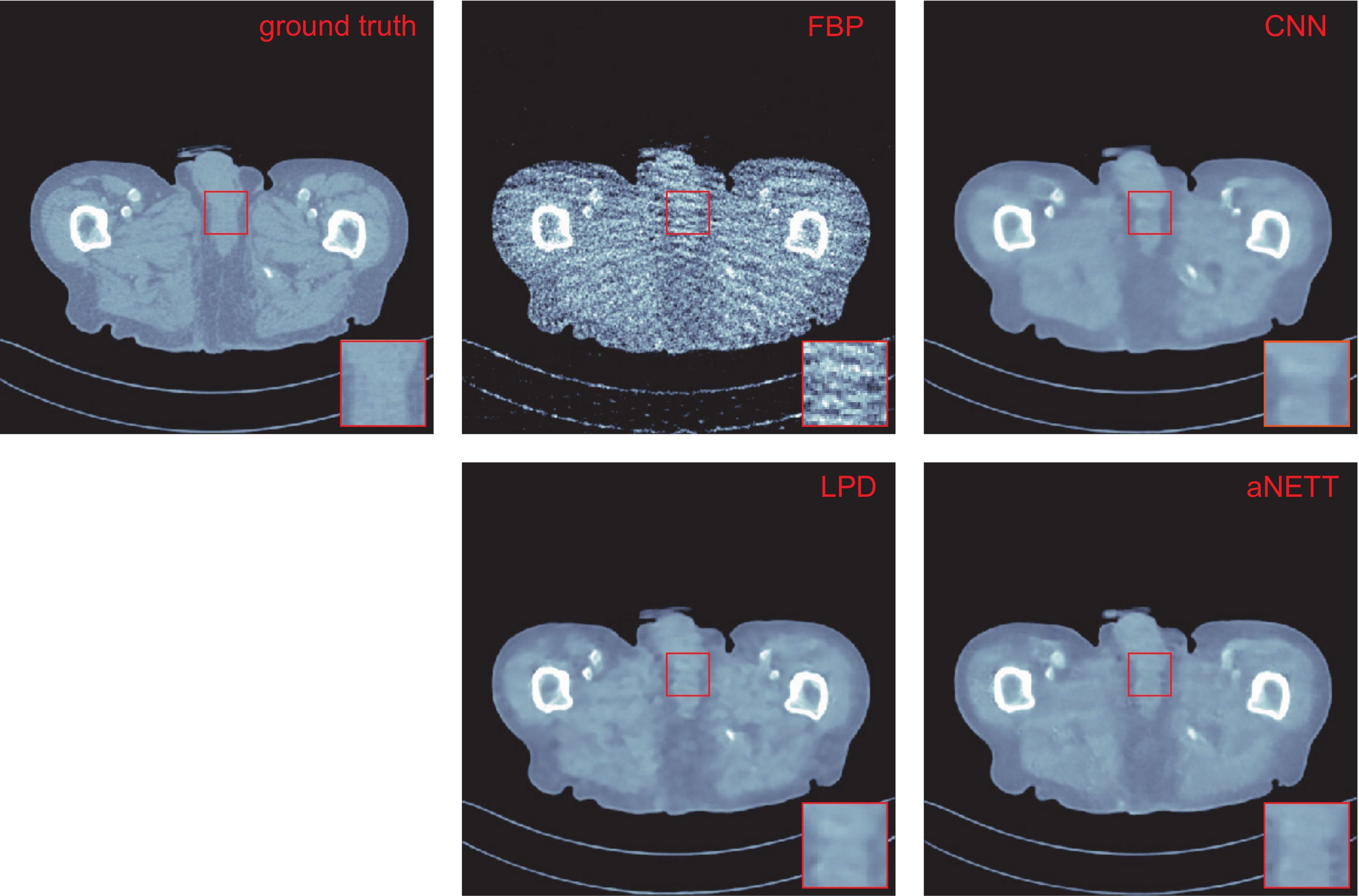}
\caption{Reconstructions for sparse view CT data from  $N_\varphi =40$ angular directions. The intensity range of all images is  $[-500,500]$ HU.}
  \label{fig:resultssparse}
\end{figure*}

\begin{itemize}[wide, topsep=0em, itemsep=0em]
\item \textsc{Sparse View CT:}
To simulate sparse view data we evaluate the  Radon transform for   $N_\varphi = 40$ directions. We generate noisy data $\data^\delta = \Ko \signal + \eta^\delta$ by adding Gaussian white   noise with standard deviation taken as  $0.02$ times the mean value of $\Ko \signal $. We use the $\ell^2$-norm distance as the similarity measure.  Quantitative results evaluated on the test set are shown in Table~\ref{tab:results}. All learning-based methods  yield comparable performance  in terms of PSNR and clearly outperform FBP. The reconstructions shown in Figure~\ref{fig:resultssparse} indicate that  aNETT reconstructions are  less smooth than CNN reconstructions and less blocky than  LPD reconstructions.

\begin{figure*}[htb!]
\centering
\includegraphics[width=\textwidth]{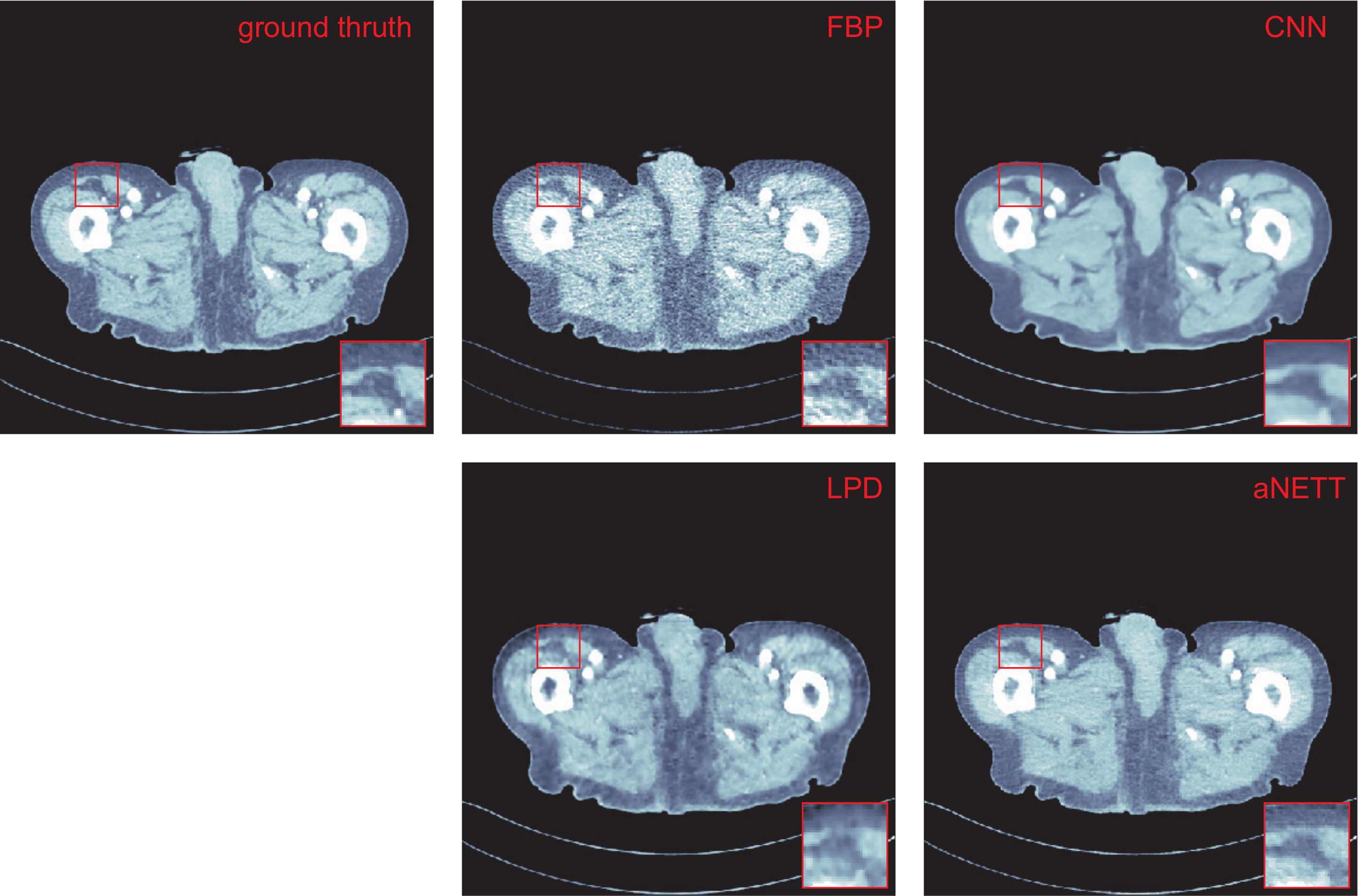}
  \caption{Reconstructions results from low dose CT data. The intensity range of all images is  $[-200,200]$ HU.}
  \label{fig:resultslow}
\end{figure*}

\item \textsc{Low Dose CT:}
For the low dose problem, we use a fully sampled sinogram with $N_\varphi = 1138$ and add Poisson noise corresponding to $10^4$ incident photons per pixel bin. The Kullback-Leibler divergence $\similarity_{\rm KL}$ is a more appropriate than the squared $\ell^2$-norm distance in case of Poisson noise and the reported values and reconstructions use the Kullback-Leibler divergence as the similarity measure. Quantitative results are shown in Table~\ref{tab:results}. Again,  all learning-based methods give similar results and significantly outperform FBP. Visual comparison of the reconstructions in Figure~\ref{fig:resultslow} shows that  CNN yields cartoon like images and the LPD reconstruction again looks blocky. The aNETT reconstruction shows more texture than the CNN reconstruction and at the same time is less blocky than the LPD  reconstruction.

\begin{figure*}[htb!]
\centering
\includegraphics[width=\textwidth]{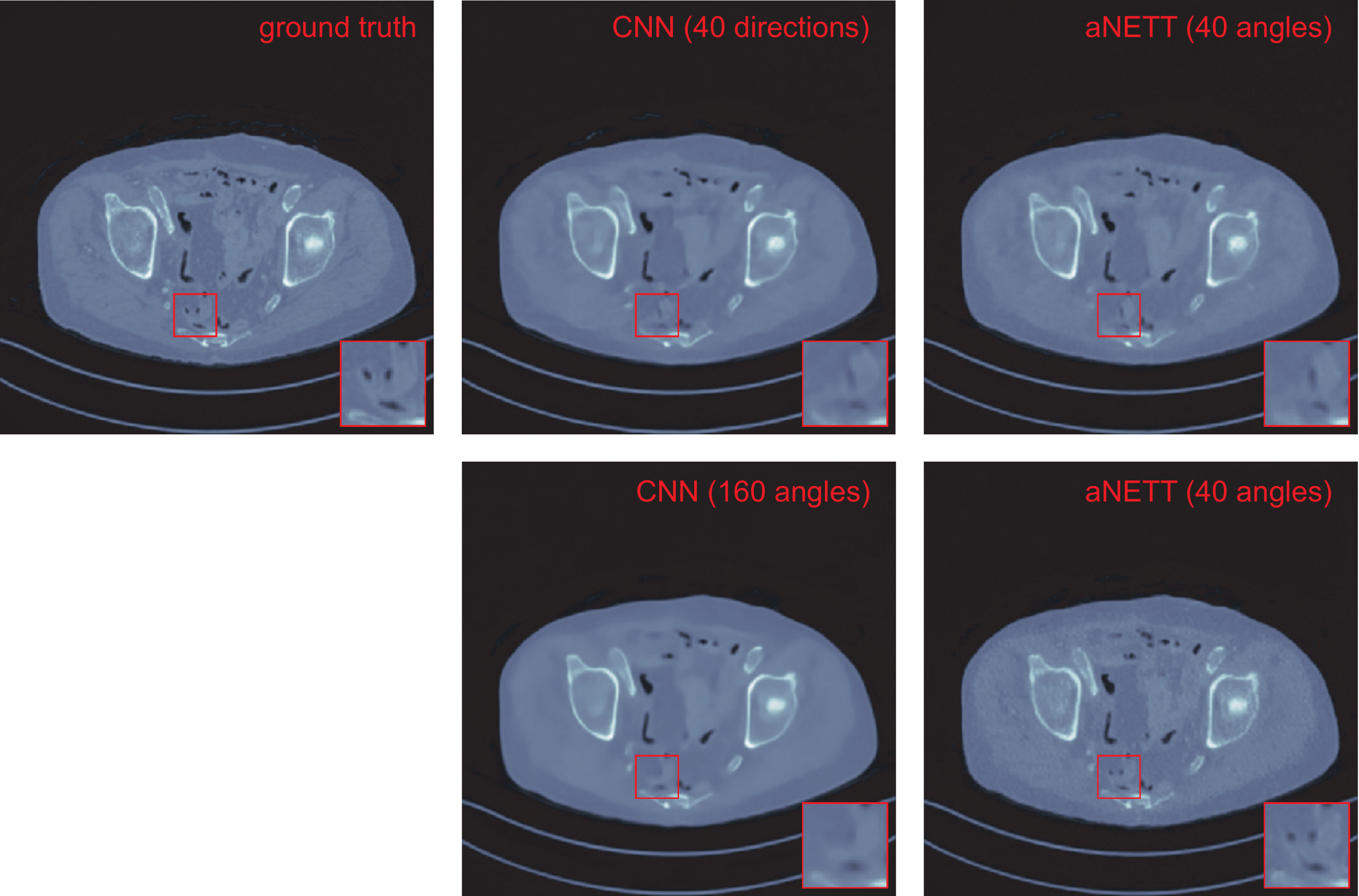}
\caption{Universality of aNETT due to change of angular sampling pattern. Top row: Ground truth and reconstructions from 40 angular directions. Bottom: Reconstructions from 160 angular directions. All reconstruction use the networks trained with 40 angular  directions. While aNETT shows increased resolution for ncreased angular sampling, CNN does not.}
\label{fig:adaptability}
\end{figure*}

\item \textsc{Universality:}
In practical applications, we may not have a fixed sampling pattern.  If we have many different sampling patterns, then training network for each sampling pattern is infeasible and hence reconstruction methods should be applicable  to different sampling scenarios. Additionally, it is desirable that an increased number of  samples indeed increases performance. In order to test this issue, we consider the sparse view CT problem but with an increased  number  of angular samples  without retraining the networks.  LPD is not applicable in this case, as the changing the forward operator changes the network architecture.   Quantitative evaluation for this scenario is given in Table~\ref{tab:results}. We see that aNETT are better  than CNN in terms of PSNR. The advantage of aNETT over CNN, however, is best observed in Figure~\ref{fig:adaptability}. One observes that  CNN  yields a similar reconstructions for both angular sampling patterns. On the other hand, aNETT is able to synergistically combine the increased  sampling rate of the sinogram with the network trained on coarsely sampled data. Despite using the network, aNETT with 160 angular samples reconstructs small details which are not present in the reconstruction from 40 angular samples.
\end{itemize}

\subsection{Discussion}

The results show that the proposed aNETT regularization is competitive with prominent deep-learning methods such as LPD and post-processing CNNs. We found that the aNETT does not suffer as much from over-smoothing which is often observed in other deep-learning reconstruction methods. This can for example be seen in Figure~\ref{fig:resultslow} where the CNN yields an over-smoothed reconstruction and the aNETT reconstruction shows more texture. Besides this, aNETT reconstructions are less blocky than LPD reconstructions.  Moreover, aNETT is able to leverage higher sampling rates to reconstruct small details while other deep-learning methods fail to do so. We conjecture that this is due the synergistic interplay of the aNETT regularizer with the data-consistency term in \eqref{eq:anett}.   
In some scenarios, it may not be possible to retrain  networks. Especially for learned iterative schemes  network training is  a time-consuming task. Training aNETT on the other hand is straightforward and, as demonstrated, yields a method which is robust to changes of the forward problem during testing time.

Finally, we note that aNETT relies on minimizing \eqref{eq:anett} iteratively.  With the use of the ADMM minimization scheme presented in this article, aNETT is slower than the methods used for comparison in this article. Designing faster optimization schemes for \eqref{eq:anett} is beyond the scope of this work, but is an important and interesting aspect.

\section{Conclusion} \label{sec:conclusion}

We have proposed the aNETT (augmented NETwork Tikhonov) for which derived coercivity of the regularizer under quite mild assumptions on the networks involved. Using this coercivity we presented a  convergence analysis of aNETT with a general similarity measure $\similarity$. We proposed a modular training strategy in which we first train an $\pen$-regularized autoencoder independent of the problem at hand and then a network which is adapted to the problem and first autoencoder. Experimentally we found this training strategy to be superior to directly training the autoencoder on the full task. Lastly, we conducted numerical simulations demonstrating the feasibility of aNETT.

The experiments show that aNETT is able to keep up with the classical post-processing CNNs and the learned primal-dual approach for to sparse view and low dose CT. Typical deep learning methods work well for a fixed sampling pattern on which they have been trained on. However, reconstruction methods are expected to perform better if we use an increased  sampling rate. We have experimentally shown that aNETT is able to leverage  higher sampling rates to reconstruct small details in the images which are not visible in the other reconstructions. This universality  can be advantageous in applications where one is not fixed to one sampling pattern or is not able to train a network for every sampling   pattern.

\section*{Acknowledgments}
D.O. and M.H.  acknowledge support of the Austrian Science Fund (FWF), project P 30747-N32.
The~research of L.N. has been supported by the National Science Foundation (NSF) Grants DMS 1212125 and DMS 1616904.

\end{document}